\documentclass[12pt]{article}

\usepackage[numbers,sort&compress]{natbib}
\setcitestyle{numbers,square,citesep={, }}

\usepackage{dcolumn}
\usepackage{bm}
\usepackage{verbatim}       

\usepackage{amsthm}
\usepackage{amssymb}

\usepackage{amstext}
\usepackage{amsmath}
\usepackage{amsfonts}
\usepackage{units}
\usepackage{mathrsfs}
\usepackage{xcolor}

\newcommand{\bd}{\begin{definition}}                
\newcommand{\ed}{\end{definition}}                  
\newcommand{\bc}{\begin{corollary}}                 
\newcommand{\ec}{\end{corollary}}                   
\newcommand{\bl}{\begin{lemma}}                     
\newcommand{\el}{\end{lemma}}                       
\newcommand{\bp}{\begin{proposition}}            
\newcommand{\ep}{\end{proposition}}                
\newcommand{\bere}{\begin{remark}}                  
\newcommand{\ere}{\end{remark}}                     

\newcommand{\bt}{\begin{theorem}}
\newcommand{\et}{\end{theorem}}

\newcommand{\be}{\begin{equation}}
\newcommand{\ee}{\end{equation}}

\newcommand{\bit}{\begin{itemize}}
\newcommand{\eit}{\end{itemize}}
\newtheorem{theorem}{Theorem}[section]
\newtheorem{corollary}[theorem]{Corollary}
\newtheorem{lemma}[theorem]{Lemma}
\newtheorem{proposition}[theorem]{Proposition}
\theoremstyle{definition}
\newtheorem{definition}[theorem]{Definition}
\theoremstyle{remark}
\newtheorem{remark}[theorem]{Remark}

\usepackage[hidelinks]{hyperref}

\title{A Stone-Weierstrass approximation theorem for monotone functions%
\thanks{\textit{Keywords:} compact ordered spaces, Gelfand representation, semi-algebras, Bernstein polynomials, causality theory.
\newline
\textit{2020 Mathematics Subject Classification:} 41A10, 41A20, 41A36, 46J10, 06F30.}}

\author{E. Minguzzi\\[3pt]
\small Dipartimento di Matematica, Universit\`a degli Studi di Pisa, Pisa, Italy\\
\small E-mail: ettore.minguzzi@unipi.it \quad ORCID: 0000-0002-8293-3802}

\date{}

\begin{document}

\maketitle

\begin{abstract}
We present an approximation theorem for continuous non-decreasing functions on compact preordered spaces, leading to an algebraic characterization of their corresponding function spaces. As an application, we prove that the family of positive non-decreasing  rational functions with non-negative coefficients  can uniformly approximate all non-negative continuous non-decreasing functions on compact intervals. An explicit approximation formula of this type is provided.
\end{abstract}

\section{Introduction}
In this work we establish a Stone-Weierstrass type theorem on the approximation of continuous isotone functions over a compact ordered space.

We recall that a {\em topological preordered space} $(X,\mathscr{T},\le)$ is a triple given by a topological space $(X,\mathscr{T})$ and a {\em preorder} $\le$ (reflexive and transitive relation) on it. The preorder is an {\em order} if it is antisymmetric, in which case the triple is called a {\em topological ordered space}. A {\em closed preordered space} is a topological preordered space for which $\le$ is closed in the product topology. Closed ordered spaces are Hausdorff.
A {\em compact preordered space} is a closed preordered space for which $(X,\mathscr{T})$ is compact. Compact preordered spaces were studied by Nachbin \cite{nachbin65,fletcher82}.

On a topological preordered space a function $f:X\to \mathbb{R}$ is said to be {\em isotone} (non-decreasing) if $x\le y \Rightarrow f(x)\le f(y)$.

On a topological preordered space, given a subset $N\subset X$, we denote $C(X,\le,\mathbb{R},N)$  the set of  continuous isotone real functions that vanish on $N$, and we denote $C^+(X,\le,\mathbb{R},N)$ the set of non-negative  continuous isotone real functions that vanish on $N$. We omit $N$ if it is empty, so writing $C(X,\le,\mathbb{R})$ and  $C^+(X,\le,\mathbb{R})$, respectively.
We endow $C(X,\mathbb{R})$ with the uniform norm, so if $S\subset C(X,\mathbb{R})$, with $\bar S$ we shall mean the closure in the uniform norm.

Given a subset $S\subset C(X,\mathbb{R})$ we denote with $N_S$ the closed subset of $X$ given by
\[
N_S=\cap_{f\in S} f^{-1}(0)=\{x\in X: f(x)=0, \ \forall f\in S\}.
\]
Given a subset $S\subset C(X,\mathbb{R})$ we denote with $\le_S$ the closed relation given by $x \le_S y$ if $f(x)\le f(y)$ for every $f\in S$, that is,
\[
\le_S=\{(x,y)\in X^2: \ f(x)\le f(y), \ \forall f\in S \}.
\]
It is an order if and only if the family $S$ distinguishes (separates) points.

A non-empty subset $F\subset C(X,\mathbb{R})$ is a {\em semi-algebra} \cite{bonsall60} if $f+g$, $\alpha f$, $fg \in F$ whenever $f,g\in F$ and $\alpha\ge 0$. It is called a {\em semi-algebra with identity} (or {\em unit}) if it contains the unit function $1$. According to Bonsall, a semi-algebra is said to be of {\em type-I} if $f/(1+f)\in F$ whenever $f\in F$. Observe that $C^+(X,\le,\mathbb{R})$  is a uniformly closed type-I semi-algebra with identity.

Let $(X,\mathscr{T},\le)$  be a compact preordered space.
We are interested in conditions on $S\subset C(X,\le,\mathbb{R})$ which ensure that every element of $C(X,\le,\mathbb{R})$ or $C^+(X,\le,\mathbb{R})$ is uniformly approximated by elements of $S$.

Every compact preordered space is a {\em normally preordered space} \cite{nachbin65,minguzzi11f} which implies that $\le=\le_{C(X,\le,\mathbb{R})}$ and hence, since every continuous function  on $X$ is bounded and can be made non-negative by adding a constant,  $\le=\le_{C^+(X,\le,\mathbb{R})}$. A necessary condition for the mentioned approximation property is then $\le=\le_S$.

Two approximation theorems of this type have appeared in the literature which establish that this condition, under suitable invariance properties of $S$, is also sufficient.

The former by Bonsall \cite[Theorem 7]{bonsall60} can be written as follows

\begin{theorem} \label{cnz}
Let $(X,\mathscr{T},\le)$ be a compact preordered space.
Let $S\subset C(X, \le, \mathbb{R})$ be a uniformly closed  type-I semi-algebra with identity such that $\le=\le_S$, then $S=C^+(X,\le, \mathbb{R})$.
\end{theorem}

The latter by Besnard \cite{besnard13} reads instead
\begin{theorem} \label{besn}
Let $(X,\mathscr{T},\le)$ be a compact ordered space.
Let $S\subset C(X, \le, \mathbb{R})$ be such that  $\le=\le_S$,  $S$ is closed under sum, and it is also closed under any map of the form $f\mapsto h(f)$ where $h: \mathbb{R} \to \mathbb{R}$ is  any continuous non-decreasing piecewise linear function. Then $\bar S=C(X,\le, \mathbb{R})$.
\end{theorem}

Taking $h$ equal to a constant shows that $S$ includes all real constants. Also taking $h=ax$, $a\ge 0$, shows that $S$ is a cone, hence convex by closure under sum.

Observe that the Theorem \ref{cnz} characterizes $C^+(X,\le, \mathbb{R})$ among its subsets, while Theorem \ref{besn} characterizes $C(X,\le, \mathbb{R})$ among its subsets.
%
%

Another difference is that Theorem \ref{cnz} is more general in dealing with preorders rather than orders.

Bonsall's type-I condition is more elegant, being more algebraic in spirit than Besnard's condition, which requires invariance under a large family of continuous, non-decreasing, piecewise linear functions. These piecewise linear functions lack analyticity due to their corners. On the other hand, Bonsall's theorem imposes a strong semi-algebra condition that is absent in Besnard's result. Additionally, Besnard's proof is simpler and more direct than Bonsall's.

This work presents a theorem that both generalizes and strengthens prior results. Following Besnard’s adaptation of the Brosowski–Deutsch proof \cite{brosowski81} of the Stone–Weierstrass theorem, our argument incorporates a crucial new element: the identification and careful application of specific fixed-point properties of a subset-preserving function $\phi$.

Curiously,  Nachbin's theory of topological ordered spaces does not really enter the proof of the following result.

\begin{theorem} \label{besm2}
Let $(X,\mathscr{T},\le)$ be a compact preordered space.
Let $\phi: [0,+\infty)\to [0,+\infty)$ be a non-decreasing continuous function such that $\phi(x)\le x$ with equality iff $x=0$ or $x=1$ (see some examples below).
Let $S\subset C^+(X,\le,\mathbb{R})$ be non-empty and such that
\begin{itemize}
\item[(i)] $S$ is a convex cone: $S+S\subset S$ and $\lambda S\subset S$ for every $\lambda > 0$,
\item[(ii)] If $f \in S$  then $\phi(f) \in S$,
\item[(iii)] $S$ generates $\le$, that is, $\le=\le_S$.
 \end{itemize}
Then $\bar S=C^+(X,\le,\mathbb{R},N_S)$.
In particular, if additionally
\begin{itemize}
\item[(iv)] $N_S=\emptyset$, that is: for every $x\in X$ there is some $f\in S$ such that $f(x)>0$ (this follows if $S$ contains the identity).
\end{itemize}
Then $\bar S=C^+(X,\le,\mathbb{R})$.
\end{theorem}

As a consequence, if $S$ is uniformly closed then it coincides with $C^+(X,\le,\mathbb{R},N_S)$ and so it is invariant under the lattice operations $\vee$ and $\wedge$.

We are not assuming that $S$ contains the (necessarily non-negative) constants.

\begin{remark}
This and the following results can be specialized to the discrete order, $\le=\Delta=\{(x,x): x\in X\}$ (which is closed, as required, precisely when $(X,\mathscr{T})$ is Hausdorff), for which they provide  already original statements. The condition $\le =\le_S$ reads ``for every $x,y\in X$, $x\ne y$, there is $f\in S$, such that $f(x)<f(y)$''. The set $C^+(X,\le,\mathbb{R})$ is the family of continuous non-negative functions.
\end{remark}

\begin{remark}
Observe that for any subset $N$, $C^+(X,\le,\mathbb{R},N)$ is uniformly closed and  satisfies (i) and (ii), while (iii) can fail for $N\ne\emptyset$: for $X=[0,1]$ with the usual order and $N=\{1\}$, every non-negative isotone function vanishing at $1$ vanishes identically, so $C^+(X,\le,\mathbb{R},N)=\{0\}$, which generates the trivial preorder $X\times X$. Actually, (i) holds also for $\lambda=0$ because $0\in C^+(X,\le,\mathbb{R},N)$. Observe also that $C^+(X,\le,\mathbb{R})$ is uniformly closed, satisfies (i)-(iv) and all non-negative constants belong to it.
\end{remark}


A useful example for function $\phi$ is $\alpha_1:=2x^2/(1+x^2)=x \operatorname{sech}(\log x)$.
There are other functions with the properties of $\phi$ including
\begin{itemize}
\item[($\alpha$)] $\alpha_a(x):=x\operatorname{sech}(a \log x)=2x^{1+a}/(1+x^{2a})$ for
$0<a\le 1$,
\item[($\beta$)]  $\beta(x):=x-\frac{x^2(1-x)^2}{(1+x)^3}=\frac{x(1+2x+5x^2)}{(1+x)^3}$,
\item[($\gamma$)] $\gamma(x):=x-\frac{x(1-x)^2}{(1+x)^2} =\Big(\frac{2x}{1+x}\Big)^{\!2}$.
\end{itemize}

For $\alpha_a$, the inequality $\alpha_a(x)\le x$ for $x>0$ is equivalent to $2x^a\le 1+x^{2a}$, with equality iff $x=1$, while monotonicity follows from
\[
\frac{d}{dx}\,\frac{2x^{1+a}}{1+x^{2a}}=\frac{2x^a\big[(1+a)+(1-a)x^{2a}\big]}{(1+x^{2a})^2}\ge 0, \qquad 0<a\le 1.
\]
Of course, it is also easy to construct a piecewise linear function $\phi$ with the desired properties.

\begin{remark}  \label{ntyu}
If $S$ is stable under squaring $f \to f^2$, e.g.\ because it is a semi-algebra, and under the map $f \to f/(1+f)$ then it is also stable under $\gamma$, which shows that Bonsall's result, Theorem  \ref{cnz}, follows  as a special case of this theorem. Here we use that the elements of a type-I semi-algebra $S\subset C(X,\le,\mathbb{R})$ are necessarily non-negative: if $f(x)<0$ for some $f\in S$ and $x\in X$, then $\lambda f(x)=-1$ for some $\lambda>0$, and $\lambda f/(1+\lambda f)$ would not be defined. Thus $S\subset C^+(X,\le,\mathbb{R})$, $N_S=\emptyset$ since $1\in S$, and $\gamma(f)=4\big(f/(1+f)\big)^2\in S$ for every $f\in S$.
\end{remark}

\begin{remark} \label{caew}
Sometimes using $\beta$ and $\gamma$ can be convenient as their singularity at $-1$ together with the cone assumption on $S$ might be used to get that $S$ consists of non-negative functions, not by assumption, but by using  the fact that $\beta$ or $\gamma$ are well defined on $S$.
\end{remark}

We denote with $\phi^{(n)}$ the composition of $\phi$ for $n$ times. It can be observed that  $\phi^{(n)}$  has fixed points $0$ and $1$ and converges (non-uniformly) to the characteristic function of $[1,+\infty)$.

We shall need a couple of lemmas. It  can be noted that we do not need to assume that $(X,\mathscr{T})$ is Hausdorff. This property would follow from a antisymmetry condition on $\le$ which we do not impose.

\begin{lemma} \label{kord}
Let $a,b\in X$ be such that $b\nleq a$ then for any $\epsilon>0$ there is $f_{ab}\in S$ such that, $0\le f_{ab}\le 1+\epsilon$, $f_{ab}(a)\in [0,\epsilon)$ and $f_{ab}(b)=1$.
\end{lemma}

\begin{proof}
Since $b\nleq a$ by (iii) we can find $f\in S$ such that $f(a)<f(b)$. Multiplying $f$ by a positive constant we can further assume that $f(b)=1$.
Since \( \phi^{(n)}(f) \) converges pointwise to the characteristic function of \( \{ f \ge 1 \} \), and \( \phi^{(n)}(x) \) is decreasing in $n$ for \( x > 1 \), the maximum of \( \phi^{(n)}(f) \) converges to 1. Thus, for large \( n \), \( \phi^{(n)}(f) \le 1 + \epsilon \). Similarly, for sufficiently large $n$, $\phi^{(n)}(f)(a)<\epsilon$, thus setting $f_{ab}=\phi^{(n)}(f)$ we conclude.
\end{proof}

\begin{lemma} \label{lcmv}
Let $A,B$ be two non-empty closed (hence compact) subsets of $X$ such that for every $a\in A$ and $b\in B$,  $b\nleq a$. Then for every $\delta$, $0<\delta<1$ there is $f_{AB}\in S$ such that $f_{AB}(A)\subset [0,\delta)$, $f_{AB}(B)\subset [1,1+\delta)$, $f_{AB}(X)\subset [0,1+\delta)$.
\end{lemma}

\begin{proof}
Let $0<\epsilon<1$. For all $a \in A$ and $b\in B$ we can find $f_{ab}\in S$ as in Lemma \ref{kord}. We fix $b\in B$ and let $a$ vary in $A$. Since $f_{ab}$ is continuous there is an open neighborhood $V_a$ of $a$ such that $f_{ab}(V_a)\subset [0,\epsilon)$. By compactness of $A$ there are $V_1,\ldots, V_k$ corresponding to $a_1,\ldots, a_k$ such that $A\subset V_1\cup \cdots \cup V_k$. For $k=1$ let $g_{b}=f_{a_1b}\in S$ so that for every $a\in A$, $g_b(a)<1$, and moreover $g_b(b)=1$. For $k\ge 2$ let $g_b:=\frac{1}{k}\sum_{1}^k \phi^{(n)}(f_{a_i b})$, where  $n$ is chosen sufficiently large that $\phi^{(n)}([0,1+\epsilon))\subset [0,\frac{k-\epsilon}{k-1})$. Note that $g_b \in S$. Moreover, $g_b(b)=1$ and for all $a\in A$,
\[
g_b(a)< \frac{1}{k} [(k-1)\frac{k-\epsilon}{k-1}+\epsilon]=1.
\]
In both cases, by compactness of $A$ there is $\alpha \in (0,1)$ such that $g_b(A)$ is a compact subset of $[0,\alpha)$. Let $\lambda:=1+\delta/2$ and $f_{Ab}:=\lambda\,\phi^{(m)}(g_b/\alpha) \in S$, where $m$ is chosen sufficiently large that $\phi^{(m)}([0,\sup g_b(A)/\alpha])\subset [0,\delta/\lambda)$ and $\phi^{(m)}(\sup_X g_b/\alpha)<(1+\delta)/\lambda$; this is possible because $\sup g_b(A)/\alpha<1$, $\sup_X g_b/\alpha\ge 1/\alpha>1$ and $(1+\delta)/\lambda>1$. Then $f_{Ab}(A)\subset[0,\delta)$, $f_{Ab}(X)\subset [0,1+\delta)$ and, since $\phi^{(m)}(1/\alpha)\ge 1$, $f_{Ab}(b)\in [\lambda,1+\delta)\subset(1,1+\delta)$. The factor $\lambda$ guarantees the strict inequality $f_{Ab}(b)>1$, which would not follow from the assumptions on $\phi$ alone, as $\phi$ might be equal to $1$ on an interval $[1,c]$, $c>1$; if $\phi(x)>1$ for $x>1$, e.g.\ if $\phi$ is strictly increasing, one can take $\lambda=1$. By continuity of $f_{Ab}$ there is an open neighborhood $U_b$ of $b$ such that $f_{Ab}(U_b)\subset (1,1+\delta)$.  By compactness of $B$ there are $U_1,\ldots, U_l$ corresponding to $b_1,\ldots, b_l$ such that $B\subset U_1\cup \cdots \cup U_l$.
Let us consider $h=\sum_1^l \phi^{(r)}(f_{A b_i})\in S$ where $r$ is chosen sufficiently large that $\phi^{(r)}([0,\delta))\subset [0,\delta/l)$. We observe that $h(A)\subset [0,\delta)$ and for every $b\in B$, $b\in U_i$ for some $i$ so
\[
h(b) \ge\phi^{(r)}(f_{Ab_i}(b))\ge 1.
\]
This proves that $h(B)\subset [1,\infty)$ and noting that $h$ is bounded on the compact set $X$,   defining $f:=\phi^{(s)}(h)\in S$ we get for sufficiently large $s$, $f(X)\subset [0, 1+\delta)$, hence $f(B)\subset [1, 1+\delta)$ (as $\phi^{(s)}([1,\infty))\subset[1,\infty)$) while keeping $f(A)\subset [0,\delta)$.
\end{proof}

\begin{proof}[Proof of Theorem \ref{besm2}]
The constant function \( 0 \) can be uniformly approximated by elements of \( S \). Indeed, since \( S \) is non-empty, take any \( f \in S \); then \( \frac{1}{n}f \in S \) for all \( n \in \mathbb{N}_+ \), and \( \frac{1}{n}f \to 0 \) uniformly.


Let us prove that if $N_S=\emptyset$ then every other positive constant can be uniformly approximated (if $N_S\ne \emptyset$ the positive constant functions are not to be approximated as they do not belong to $C^+(X,\le,\mathbb{R},N_S)$, as they do not vanish on $N_S$). Indeed, for every $x\in X$ we can find a function $g_x\in S$ such that $g_x(x)>0$, hence by continuity, there is a neighborhood $U_x$ of $x$ such that $g_x>0$ over $U_x$. Let $U_1, \ldots, U_k$ be a finite covering of $X$ relative to points $x_1, \ldots, x_k$. Let $g=\sum_k g_{x_i}\in S$, then $g>0$ over $X$ and thus it has a minimum $m>0$ and a maximum $M\ge m$. Let $\lambda >0$ be a constant.
Let $h=\lambda \phi^{(n)}(g/m)$ then for any $\epsilon$ we have for sufficiently large $n$, $h(X)\subset [\lambda ,\lambda+\epsilon)$ which proves that the constant $\lambda$ is uniformly approximated (actually from above, and this was also the case for the zero function since the elements of $S$ are non-negative).

Let  $f\in C^+(X,\le,\mathbb{R}, N_S)$ and let $m$ and $M$ be its minimum and maximum over the compact set $X$. If $m=M$  then either $m=0$ which is a constant function that has been already shown to be approximable from above, or $m>0$ which implies that $f>0$ over $X$ and so $N_S=\emptyset$, which implies, by the previous argument, that any constant, and so the constant function $m$, is approximable from above.

We can assume that $0\le m<M$ and hence, rescaling, we can work with  $\tilde f=f/(M-m)$  instead of $f$ (since $S$ is a cone, an arbitrary uniform approximation of $\tilde f$  leads to a uniform approximation of $f$). Hence we can assume $f(X)\subset[m,m+1]$, with $m=\min f$ and $m+1=\max f$, without loss of generality. Let $n\in \mathbb{N}\backslash\{0,1\}$.
We want to prove that there is $F\in S$ such that $\Vert f-F\Vert<3/n$.

We set $A_i=f^{-1}([m,m+\frac{i}{n}])$, $B_i=f^{-1}([m+\frac{i+1}{n}, m+1])$ for each $i=0,1, \ldots, n-1$. Since $f$ is continuous and $X$ is compact, the sets $A_i,B_i$ are closed hence compact. They are non-empty because the minimum $m$ and the maximum $m+1$ of $f$ are attained. Note that if $a\in A_i$ and $b\in B_i$ then $f(b)>f(a)$ thus since $f$ is non-decreasing it cannot be $b\le a$.
For each $i$ by Lemma  \ref{lcmv} there is $f_i \in S$ such that $f_i(A_i)\subset [0,1/n)$, $f_i(B_i)\subset [1,1+1/n)$, $f_i(X)\subset [0,1+1/n)$. Let $g\in S$ be a function that approximates the constant $m$ from above, $m\le g<m+1/n$. (this $g$ exists because, for $m=0$ it was shown previously, while if $m>0$ then, $f>0$ over $X$ and so  $N_S=\emptyset$ and every constant is approximable from above as  observed previously.)
Let us consider the function $F= g+ \frac{1}{n}\sum_{i=0}^{n-1} f_i$ which  belongs to $S$.

Let $x\in X$. There are two cases to consider, depending on the value of $f(x)$.

Suppose that for some $j\in\{0, 1, \ldots, n-1\}$, $m+\frac{j}{n}< f(x)<m+\frac{j+1}{n}$.
We have $x\in A_i$ for $j+1\le i\le n-1$ (which implies $f_i(x) \in [0, 1/n)$); we have $x\in B_i$ for $i\le j-1$ (which implies  $f_i(x) \in [1, 1+1/n))$, and we have $f_j(x) \in [0, 1+1/n)$ (no restriction).
Hence with obvious meaning of the notation
\begin{align*}
F(x)&=g(x)+\frac{1}{n} \sum_{i=0}^{n-1}  f_i(x) &\\
& \in \ [m, m+\tfrac{1}{n})+ \tfrac{1}{n} \big\{ (n-1-j)[0, \tfrac{1}{n}) + [0,1+\tfrac{1}{n})+ j [1,1+\tfrac{1}{n})\big\}\\
& \subset [m+\tfrac{j}{n}, m+\tfrac{1}{n}+\tfrac{1}{n^2}(n-1-j)+\tfrac{1}{n}(1+\tfrac{1}{n})(1+j) )\\
&\subset  [m+\tfrac{j}{n}, m+\tfrac{j+1}{n}+\tfrac{2}{n} )
\end{align*}
which, since $f(x)$ belongs to this same interval, proves that $\vert f(x)-F(x)\vert<3/n$.

Suppose that for some $j\in\{0, 1, \ldots, n\}$, $f(x)=m+\frac{j}{n}$. We have $x\in A_i$ for $j\le i \le n-1$ and $x\in B_i$ for $i\le j-1$, thus
\begin{align*}
F(x)&=g(x)+\frac{1}{n} \sum_{i=0}^{n-1}  f_i(x) \\
& \in \ [m, m+\tfrac{1}{n})+ \tfrac{1}{n} \big\{ (n-j)[0, \tfrac{1}{n})  + j [1,1+\tfrac{1}{n})\big\}\\
& \subset [m+\tfrac{j}{n}, m+\tfrac{1}{n}+\tfrac{1}{n^2}(n-j)+\tfrac{1}{n}(1+\tfrac{1}{n})j )\\
&\subset  [m+\tfrac{j}{n}, m+\tfrac{j+1}{n}+\tfrac{1}{n} )
\end{align*}
which, since $f(x)$ belongs to this same interval, proves that $\vert f(x)-F(x)\vert<2/n<3/n$.
\end{proof}

As a specialization we have the following result that selects some notable functions for $\phi$. Moreover, it builds the closed order directly from $S$, more in the spirit of Bonsall's paper and of algebraic approaches.

\begin{corollary} \label{besm3}
Let $(X,\mathscr{T})$ be a compact space. Let $S$ be a non-empty subset of  the set $C(X,\mathbb{R})$. Suppose
\begin{itemize}
\item[(a)] $S$ is a convex cone: $S+S\subset S$, $\lambda S\subset S$ for every $\lambda >0$,
\item[(b)] One of the following non-excluding possibilities:
\begin{enumerate}
\item $S$ consists of non-negative functions and if  $f\in S$ then $\chi(f):=\frac{2f^2}{1+f^2}$ belongs to $S$,
\item  If $f\in S$ then $\gamma(f):=\Big(\frac{2f}{1+f}\Big)^{\!2}$ is well defined and belongs to $S$.
\item  $S$ consists of positive functions, it contains the positive constants and if $f\in S$ then  both $f^2$ and $f/(1+f)$ belong to $S$. \label{mnnf}
\end{enumerate}

\end{itemize}
Then  $\bar S=C^+(X,\le_S,\mathbb{R}, N_S)$.
In particular, if additionally
\begin{itemize}
\item[(c)] $N_S= \emptyset$, that is, for every $x\in X$ there is some $f\in S$ such that $f(x)>0$ (this follows if $S$ contains the identity, e.g.\ in case b3),
\end{itemize}
then  $\bar S=C^+(X,\le_S,\mathbb{R})$.
\end{corollary}

Note that (b)3 implies $N_S=\emptyset$ because $1\in S$.

\begin{proof}
It is clear that $\le_S$ is a closed preorder so $(X,\mathscr{T},\le_S)$ is a compact preordered space. As mentioned in Remark \ref{caew} if there were $f\in S$ such that for some $x$, $f(x)<0$ then for some positive constant $\lambda>0$, $f'=\lambda f$, would be such that $f'(x)=-1$, but then $\gamma(f')$ would not be well defined. We conclude that in case 2 the set $S$ consists of non-negative functions.
We just need to prove $3\Rightarrow 1$. If $f\in S$ then $f^2\in S$, hence $f^2/(1+f^2)\in S$ and, by the cone property, $\chi(f)=2f^2/(1+f^2)\in S$. The functions in $S$ are positive thus they are non-negative and so 1 is proved.
\end{proof}



We have already observed (Remark \ref{ntyu}) that Theorem \ref{besm2} implies Bonsall's.
Let us prove that Besnard's theorem also follows.

\begin{proposition}
Theorem \ref{besm2} implies Besnard's  Theorem \ref{besn}.
\end{proposition}

\begin{proof}
Under the assumption of Besnard's theorem, consider the projections $p_+(x)=\max\{x,0\}$ and $p_-=\min\{x,0\}$, and define $S^+=p_+(S)$, $S^-=p_-(S)$. The functions $p_+, p_-:\mathbb{R} \to \mathbb{R}$ are piecewise linear non-decreasing functions, thus $p_+(f),p_-(f)\in S$ for $f\in S$.
The set $S^+$ is closed under sum, because of $0\le p_+(f)+p_+(g)\in S$, thus $p_+(f)+p_+(g)=p_+(p_+(f)+p_+(g))\in S^+$. Moreover, for $a\ge 0$, $p_+(af)=ap_+(f)$ thus $S^+$ is a convex cone. But we can find a piecewise linear function $ \phi: [0,\infty)\to [0,\infty)$ which respects the conditions of Theorem \ref{besm2}, and extending it to a function $\check \phi$ equal to zero  for $x<0$, we have that $\check \phi$ is piecewise linear and so for $f\in S$, $\phi(p_+(f))=\check\phi(f)=p_+(\check \phi(f))\in S^+$ which means that $S^+$ is invariant under $\phi$. Moreover, $S^+$ generates $\le$. Indeed, $\le\,\subset\,\le_{S^+}$ because the elements of $S^+$ are isotone. Conversely, if $x\nleq y$, by $\le=\le_S$ there is $f\in S$ with $f(x)>f(y)$; the function $h(t)=\max\{t-f(y),0\}$ is continuous, non-decreasing and piecewise linear, so $h(f)\in S$, and since $h(f)\ge 0$ we have $h(f)=p_+(h(f))\in S^+$, while $h(f)(x)>0=h(f)(y)$. Thus $x\nleq_{S^+} y$. Finally, $N_{S^+}=\emptyset$ because $S$ contains the constants (take $h$ constant), so that $1=p_+(1)\in S^+$. Thus, by Theorem \ref{besm2} $\overline{S^+}=C^+(X,\le,\mathbb{R})$. Now consider $S^-$, then $-S^-\subset C^+(X,\ge,\mathbb{R})$ where $\ge$ is the transposed order. Arguing as above
we get $\overline{-S^-}=C^+(X,\ge,\mathbb{R})$ and so $\overline{S^-}=-\overline{-S^-}=-C^+(X,\ge,\mathbb{R})=C^-(X,\le,\mathbb{R})$.
 Finally,
\[
\overline{S}\supset \overline{S^+}+\overline{S^-}={C^+(X,\le,\mathbb{R})}+{C^-(X,\le,\mathbb{R})}={C(X,\le,\mathbb{R})},
\]
The other inclusion being clear.
\end{proof}

The following result could also be obtained from Bonsall's theorem. Indeed, the family $S$ below is closed under products and contains $1$, and $f/(1+f)=p/(p+q)$ belongs to $S$ whenever $f=p/q\in S$, so $S$ is a type-I semi-algebra with identity. These properties pass to the uniform closure $\bar S$ (for the type-I property note that $t\mapsto t/(1+t)$ is $1$-Lipschitz on $[0,\infty)$), and $\le_{\bar S}=\le_S$, so Theorem \ref{cnz} applies to $\bar S$. The same remark applies to Theorem \ref{cptr2} below. We prefer to deduce both results from Corollary \ref{besm3}, which only requires the stability under $\chi$.
\begin{theorem} \label{cptr}
Let $\mathbb{R}^n$ be endowed with the Cartesian  product order: $x\le y$ iff $x_k\le y_k$ for $k=1,\ldots, n$. Let $X\subset  \mathbb{R}^n$ be a compact subset endowed with the induced topology and order (denoted in the same way).
Let $S$ be the family of positive isotone rational functions $f:=p(x)/q(x)$  on $X$ (that is, $q$ does not vanish on $X$ and $p/q>0$ on $X$) such that (i) the coefficients of $p$ and $q$ are non-negative, (ii) $\textrm{deg}(p) = \textrm{deg}(q)$, where $\deg$ denotes the total degree.
 Then $\bar S=C^+(X,\le,\mathbb{R})$.
\end{theorem}


Conditions (ii)  can be dropped but its presence makes the theorem somewhat stronger.

\begin{remark}
For $n=1$, the theorem states  that every non-negative continuous isotone function on $[a,b]\subset \mathbb{R}$ can be uniformly approximated on $[a,b]$ by rational functions $p(x)/q(x)$ that are positive, isotone  and with non-negative coefficients. An explicit approximation of this type is given in Section \ref{secber}.
\end{remark}
\begin{proof}
To start with $\le \subset \le_S$ because all the elements of $S$ are isotone. Let $x\le_S y$, we want to prove that $x\le y$ and hence $\le =\le_S$ due to the arbitrariness of $x$ and $y$. If not $y_k<x_k$ for some $k$. Now consider the function $g: X \to \mathbb{R}$,  $z \mapsto g(z)=\frac{z_k+c}{z_k+c+1}$, $c>0$. Since $z_k$ is continuous over $X$ it reaches a minimum and so $c>0$ can be chosen so that $z_k+c$ is positive on $X$. With this choice $g\in S$ thus $x \le_S y$ implies $g(x)\le g(y)$ which reads $x_k\le y_k$, a contradiction. We proved $\le =\le_S$.

Let us prove that $S$ satisfies (a), (b)1, (c) of Cor.\ \ref{besm3}. Condition (a) is satisfied because the product or sum of two polynomials with non-negative coefficients gives a polynomial with non-negative coefficients. 
The condition on the degrees is easily proved noting that  $\frac{p_1}{q_1}+\frac{p_2}{q_2}=\frac{p_1 q_2+q_1p_2}{q_1q_2}$ and $\textrm{deg}(p_1)=\textrm{deg}(q_1)$, $\textrm{deg}(p_2)=\textrm{deg}(q_2)$ so both terms in the numerator have the same degree of the denominator. The degree cannot get reduced by the sum on the numerator because, even if the highest degree terms of two addenda $p_1q_2$, $q_1 p_2$ involve the same variables, the sum cannot vanish as the coefficients of the polynomials are non-negative.


Condition (b)1 follows from the fact that the composition of isotone functions is  isotone so if $f\in S$, $\chi(f)=2f^2/(1+f^2)$ is also isotone. Note that since $f$ is positive $\chi(f)$ is positive, and also if $f=p/q$ since $\chi(f)=\frac{2p^2}{p^2+q^2}$ it is also rational with non-negative coefficients.
Moreover since $\deg(p)=\deg(q)$, $\deg(p^2)=\deg(p^2+q^2)$, by the same argument on the non-negativity of the coefficients used above. In conclusion, $\chi(f)\in S$.
Condition (c), $N_S=\emptyset$, follows because $1\in S$.

We conclude from Cor.\ \ref{besm3} that $\bar S=C^+(X,\le,\mathbb{R})$.
\end{proof}

In the literature the subject of approximation of continuous isotone functions by isotone polynomials is well explored \cite{lorentz68} \cite[Chapter 10, Theorem 3.3(i)-(ii)]{devore93}. In those studies the polynomial can have negative coefficients while in our version we use rational functions but restrict the coefficients to the non-negative real numbers.

\begin{theorem} \label{cptr2}
 Let $X\subset  \mathbb{R}^n$ be a compact subset endowed with the induced topology.
Let $S$ be the family of positive  rational functions $f:=p(x)/q(x)$  on $X$ (that is, $q$ does not vanish on $X$ and $p/q>0$ on $X$) such that  the coefficients of $p$ and $q$ are non-negative.
 Then $\bar S=C^+(X,\mathbb{R})$.
\end{theorem}

\begin{proof}
Let $\le$ be the discrete order on $X$. It is closed because the induced topology on $X$ is Hausdorff. Every function is isotone with respect to the discrete order, thus $\le \subset \le_S$. Let $x\le_S y$, we want to prove that $x= y$ and hence $\le =\le_S$ due to the arbitrariness of $x$ and $y$. If not $x_k\ne y_k$ for some $k$. Now, consider the functions $g,h: X \to \mathbb{R}$,  $z \mapsto g(z)=\frac{z_k+c}{z_k+c+1}$, $c>0$ and $z \mapsto h(z)=\frac{1}{z_k+c}$. Since $z_k$ is continuous over $X$ it reaches a minimum and so $c>0$ can be chosen so that $z_k+c$ is positive on $X$. With this choice $g,h\in S$ thus $x \le_S y$ implies $g(x)\le g(y)$ and $h(x)\le h(y)$ which read $x_k\le y_k$ and $y_k\le x_k$, respectively, hence $x_k=y_k$, a contradiction. We proved $\le =\le_S$.

Let us prove that $S$ satisfies (a), (b)1, (c) of Cor.\ \ref{besm3}. Condition (a) is satisfied because the product or sum of two polynomials with non-negative coefficients gives a polynomial with non-negative coefficients.

Condition (b)1 follows since if $f$ is positive $\chi(f)$ is positive, and also if $f=p/q$ since $\chi(f)=\frac{2p^2}{p^2+q^2}$ it is also rational with non-negative coefficients. In conclusion, $\chi(f)\in S$.
Condition (c), $N_S=\emptyset$, follows because $1\in S$.

We conclude from Cor.\ \ref{besm3} that $\bar S=C^+(X,\mathbb{R})$.
\end{proof}

\section{Algebraic characterization of $C^+(X,\le,\mathbb{R})$} \label{secalg}

On a compact preordered space $(X, \mathscr{T},\le)$ the family of differences of continuous isotone functions
\[
C_{BV}(X,\le,\mathbb{R}):=\{f-g, \ f,g \in C(X,\le, \mathbb{R})\}=C(X,\le,\mathbb{R})-C(X,\le,\mathbb{R})
\]
 coincides with the family of differences of non-negative continuous isotone functions $C_{BV}(X,\le,\mathbb{R})=C^+(X,\le,\mathbb{R})-C^+(X,\le,\mathbb{R})$, indeed it is sufficient to add to both $f$ and $g$ the same sufficiently large constant, to make them non-negative.

\begin{proposition} \label{vks}
Let $(X,\mathscr{T},\le)$ be a compact ordered space.
The set $C_{BV}(X,\le,\mathbb{R})$ is an algebra.
Moreover, we have the identity $\overline{C_{BV}(X,\le,\mathbb{R})}=C(X,\mathbb{R})$, namely every continuous function can be uniformly approximated by differences of continuous isotone functions.
\end{proposition}

The proof follows from the fact that the product of two non-negative continuous isotone functions is again non-negative continuous and isotone,
\[
C^+(X,\le,\mathbb{R}) \,C^+(X,\le,\mathbb{R}) \subset C^+(X,\le,\mathbb{R}).
\]
\begin{proof}
The set $C_{BV}(X,\le,\mathbb{R})$ is an algebra, indeed if $f_1,f_2,g_1,g_2\in C^+(X,\le,\mathbb{R})$
\[
(f_1-g_1) (f_2-g_2)=(f_1f_2+g_1g_2)-(g_1 f_2+f_1 g_2)
\]
where the functions in parentheses on the right-hand side belong to $C^+(X,\le,\mathbb{R})$. Since $C(X,\le,\mathbb{R})$ includes the constants and separates points (by complete order regularity \cite{nachbin65}), so does $C_{BV}(X,\le,\mathbb{R})$. The claim now follows from Stone Theorem \cite{stone48} \cite{willard70}.
\end{proof}

The following result establishes an algebraic
characterization of $C^+(X,\le,\mathbb{R})$. A result of the same type but demanding stability under the whole family of non-decreasing  continuous functions and focused on $C(X,\le,\mathbb{R})$ can be found in Besnard \cite[Theorem 4]{besnard15} \cite[Theorem 1]{besnard15b} (improving a previous approach by the same author \cite{besnard09} that also demanded for stability under lattice operations). The advantage of working with $C^+(X,\le,\mathbb{R})$ instead of $C(X,\le,\mathbb{R})$ is that the former set constitutes a semi-algebra.

An involution on an algebra $A$ over $\mathbb{C}$ is a conjugate linear map $f \to f^*$, that is $(f+a g)^*=f^*+\bar a g^*$ for $a\in \mathbb{C}$, such that $f^{**}=f$ and $(fg)^*=g^* f^*$. The pair $(A,*)$ is called a $*$-algebra. An element is {\em self-adjoint}  or {\em hermitian} if $f=f^*$. For each $f\in A$ there exist unique hermitian elements $g,h\in A$ such that $f=g+i h$.

 The element $f\in A$ is {\em normal} if $f f^*=f^*f$ in which case it generates a commutative *-algebra. If $A$ is unital $1=1^*$ because $1^*=1 (1^*)=(1(1^*))^*=((1^*))^*=1$.
A homomorphism of *-algebras $\varphi:A\to B$, that preserves adjoints $\varphi(f^*)=\varphi(f)^*$ is a *-homomorphism. If it is bijective,  it is a *-isomorphism.

A Banach *-algebra $B$ is a *-algebra with a complete submultiplicative norm such that $\Vert f^*\Vert = \Vert f\Vert$. If it has unit such that $\Vert e\Vert=1$ we call it a unital  Banach *-algebra $B$. A $C^*$-algebra is a   Banach *-algebra $B$ such that
\[
\Vert f^*f\Vert=\Vert f\Vert^2.
\]
A bijective *-homomorphism $\varphi:A\to B$ such that $\Vert\varphi(f)\Vert=\Vert f \Vert$ is a *-isometry. A character of a unital complex $C^*$-algebra $B$ is a nonzero multiplicative linear functional $x\colon B\to\mathbb{C}$ (equiv.\ nonzero  algebra homomorphisms $B\to\mathbb{C}$).

\begin{theorem} \label{cpp}
Let $B$ be a commutative  unital complex $C^*$-algebra. Let $H$ be the subset of hermitian elements, and  assume there is a subset $C\subset B$ with the following properties:
\begin{itemize}
\item[(i)] $C + C \subset C$, and $a C \subset C$  for every $a > 0$,
\item[(ii)] $C$ is  closed,
\item[(iii)] $\overline{C-C}=H$,
\item[(iv)]  We have one of the following
\begin{enumerate}
\item $C$ consists of non-negative elements (non-negative spectrum) and if  $f\in C$ then $\chi(f):=\frac{2f^2}{1+f^2}$ belongs to $C$,
\item  If $f\in C$ then $\gamma(f):=\Big(\frac{2f}{1+f}\Big)^{\!2}$ is well defined and belongs to $C$.
\end{enumerate}
\end{itemize}
Let $X$ be the set of characters. Endow $X$ with the weak*-topology $\mathscr{T}$ and with the preorder $x\le y$ if $x(f)\le y(f)$ for every $f\in C$. Then $(X,\mathscr{T},\le)$ is a compact ordered space.
Moreover, the Gelfand map $\mathcal{G}:B \to C(X, \mathbb{C})$, $f \mapsto \hat f$, $\hat f(x):=x(f)$, is a bijective unit-preserving  *-isometry  which maps bijectively $C$ to $C^+(X,\le,\mathbb{R})$.
\end{theorem}

Observe that properties $(iv)1$ and $(iv)2$ are satisfied if the maps $f \mapsto f^2$ and  $f\mapsto \frac{f}{1+f}$ send $C$ to itself (indeed $\gamma(f)=4\big(\frac{f}{1+f}\big)^2$ and $\chi(f)=2\frac{f^2}{1+f^2}$; moreover, if $f/(1+f)$ is well defined for every $f\in C$, then $C$ consists of non-negative elements, by the argument given in the proof below for (iv)2), where the former condition $C^2\subset C$ would follow if $C$ were a semi-algebra. Note also that we do not assume that the identity is in $C$,  that $C$ is a semi-algebra, or that it is stable under lattice operations. However, these properties are recovered from the proved isometry and the properties of $C^+(X,\le,\mathbb{R})$.
\begin{remark}
Of course, a more general version of the theorem holds with (iv) involving stability under a function $\phi$ with the properties of Theorem \ref{besm2}. Here we just wanted to give a more concrete realization of the result.
\end{remark}
\begin{proof}
Observe first that $C\subset H$: if $f\in C$ then $2f\in C$ by (i), thus $f=2f-f\in C-C\subset H$, so $f\in H$.
From Banach-Alaoglu the topology $\mathscr{T}$ of $X$ is Hausdorff and compact. Since for $f\in C$ the function $\hat f: X\to \mathbb{R}$ is real \cite[Theorem 2.1.9]{murphy90} and continuous the preorder $\le$ is closed. Suppose $x\le y$ and $y\le x$, then for every $f\in C$,  $x(f)\le y(f)$  and  $y(f)\le x(f)$, that is $(y-x)\vert_C=0$. But since $y-x:B\to \mathbb{C}$ is real-linear $(y-x)\vert_{C-C}=0$ and since it is continuous, by (iii), $(y-x)\vert_H=0$. Again by the complex-linearity $y-x=0$. This proves that $\le$ is an order and $(X,\mathscr{T},\le)$ is a compact ordered space.

By the Gelfand-Naimark theorem there is a bijective *-isometry $\mathcal{G}:B \to C(X, \mathbb{C})$, thus we need only to show that $\mathcal{G}(C)$ coincides with $C^+(X,\le,\mathbb{R})$.  Since $\mathcal{G}$ is linear if $\hat f,\hat g\in \mathcal{G}(C)$, $a>0$, by (i) $a \hat f\in  \mathcal{G}(C)$ and   $\hat f+\hat g\in \mathcal{G}(C)$.

Since $\mathcal{G}$ is an isometry, if $\hat f_k \to f$ uniformly in $C(X,\mathbb{R})$, then regarding $f_k, f$ as functions in $C(X,\mathbb{C})$ we have $\hat f_k \to f=\hat h$ uniformly, where $f$ belonging to $C(X,\mathbb{C})$  is really the image of some $h\in B$. Since $\mathcal{G}$ is an isometry $f_k\to h$, and since by (ii) $C$ is closed, $h\in C$, which implies $f\in \mathcal{G}(C)$, that is $\mathcal{G}(C)$ is closed.

Let $x\in X$ and suppose that $\hat f(x)=0$ for every  $\hat f\in \mathcal{G}(C)$, then $x(f)=0$ for every $f\in C$, hence $x(g)=0$ for every $g\in C-C$ and finally, since $x$ is continuous, $x\vert_H=0$ by (iii), $H=\overline{C-C}$, and finally $x=0$ by the fact that every element of $B$ can be written in the form $b=h_1+i h_2$ for $h_1,h_2\in H$. This is a contradiction since characters are non-zero by definition, thus we conclude that $N_{\mathcal{G}(C)}=\emptyset$.

Condition (iv)2 implies that the spectrum of any $f\in C$ is non-negative, for $-1$ does not belong to the spectrum of $a^{-1}f$ for $a>0$ and so $-a$ does not belong to the spectrum of $f$.

But Gelfand calculus shows that for $f\in \mathcal{G}(C)$, $\mathcal{G}(\chi(f))=\chi(\hat f)$ in case 1, and $\mathcal{G}(\gamma(f))=\gamma(\hat f)$ in case 2, thus $\mathcal{G}(C)$ is in the domain of the operation and is preserved by it in each case. The assumptions of Thm.\ \ref{besm2} are satisfied, and using the closure of $\mathcal{G}(C)$ we get $\mathcal{G}(C) =C^+(X,\le,\mathbb{R})$.
\end{proof}

\begin{remark}
Properties $(i)-(iv)$ do not really require that $B$ is commutative thanks to functional calculus on $C^*$-algebras. They can be used to define a special type of cone on $B$. The theorem then implies that the category of  commutative  unital complex $C^*$-algebras endowed with such a cone (with unital *-homomorphisms preserving the cones as morphisms) is dually equivalent to that of compact ordered spaces (with continuous isotone maps as morphisms), see Thm.\ \ref{equivalence_theorem} and Prop.\ \ref{morph}.
\end{remark}
%
%

\begin{theorem} \label{equivalence_theorem}
There is a {\em one-to-one correspondence} (a duality; see Prop.\ \ref{morph} for the morphisms) between:
\begin{enumerate}
    \item Compact ordered spaces $(X, \mathscr{T}, \le)$, and
    \item Pairs $(B, C)$, where $B$ is a commutative unital complex $C^*$-algebra and $C$ is a subset of $B$ satisfying the conditions (i)-(iv) of Theorem \ref{cpp}.
\end{enumerate}
This correspondence is implemented by the following mutually inverse (up to canonical isomorphism) structure-preserving maps:

\begin{itemize}
    \item[] {Map I (Algebra $\to$ Space):} Given $(B, C)$, construct the space $(X, \mathscr{T}, \le)$ where $X$ is the Gelfand spectrum (set of characters) of $B$, $\mathscr{T}$ is the weak*-topology, and the order $\le$ is defined by $x \le y \iff x(f) \le y(f)$ for all $f \in C$.
    \item[]{Map II (Space $\to$ Algebra):} Given $(X, \mathscr{T}, \le)$, construct the pair $(B, C)$ where $B = C(X, \mathbb{C})$ (the continuous complex functions) and $C = C^+(X, \le, \mathbb{R})$ (the non-negative continuous isotone real functions).
\end{itemize}
\end{theorem}

\begin{proof}
We must show that Map I and Map II are well-defined and that they are mutually inverse (up to canonical isomorphism).
\begin{enumerate}
\item Map I is Well-Defined.\\
Given a pair $(B, C)$, Theorem \ref{cpp} proved that the constructed triple $(X, \mathscr{T}, \le)$ is indeed a compact ordered space. Furthermore, the Gelfand map $\mathcal{G}$ is a unit-preserving $\ast$-isometry that establishes a canonical isomorphism:
\[
\mathcal{G}: (B, C) \to (C(X, \mathbb{C}), C^+(X, \le, \mathbb{R})).
\]
\item Map II is Well-Defined.\\
Given a compact ordered space $(X, \mathscr{T}, \le)$, we must show that the pair $(B, C)$, where $B = C(X, \mathbb{C})$ and $C = C^+(X, \le, \mathbb{R})$, satisfies the required conditions (i)-(iv) of Theorem \ref{cpp}.
\begin{itemize}
    \item[]{(i) Convex Cone:} If $f, g \in C$, then $f+g \in C$ and $\lambda f \in C$ for $\lambda > 0$, by the definition of non-negative isotone functions.
    \item[]{(ii) Closed:} $C\subset C(X, \mathbb{R})$ is obtained imposing  to its elements the condition $f\ge 0$ and for each $x,y\in X$, $x\le y$, the conditions $f(x)\le f(y)$. These conditions are uniformly closed so is their intersection $C$. Since $C$ is a uniformly closed subset of $C(X, \mathbb{R})$, it is also a uniformly closed subset of $C(X, \mathbb{C})$.
    \item[]{(iii) $\overline{C-C}=H$:} The set of hermitian elements $H$ in $B=C(X, \mathbb{C})$ is $C(X, \mathbb{R})$. By Prop.\ \ref{vks} the uniform closure of the set of differences of continuous isotone functions equals $C(X, \mathbb{R})$. Thus, $\overline{C-C} = C(X, \mathbb{R}) = H$.
    \item[]{(iv) $\phi$-Invariance:} For any continuous non-decreasing function $\phi: [0, \infty) \to [0, \infty)$ (satisfying the properties of Theorem \ref{besm2}), if $f \in C$, then $\phi(f)$ is also continuous, non-negative, and isotone. Since conditions (iv)1 and (iv)2 rely on specific functions $\chi$ and $\gamma$ which are of this type, $C$ is closed under these operations.
\end{itemize}
Thus, Map II is well-defined.
\item Mutually Inverse Property.\\
We verify the composition of the maps.
\begin{itemize}
    \item[]{Map II $\circ$ Map I:} Starting with $(B, C)$ and applying Map I followed by Map II results in the structure $(C(X, \mathbb{C}), C^+(X, \le, \mathbb{R}))$. By the result of Theorem \ref{cpp}, this resulting structure is canonically $\ast$-isomorphic to the starting structure $(B, C)$.
    \item[]{Map I $\circ$ Map II:} Starting with $(X, \mathscr{T}, \le)$, Map II yields
        \[
        (B', C') = (C(X, \mathbb{C}), C^+(X, \le, \mathbb{R})).
         \]
         Map I takes the Gelfand spectrum of $B'$. By the Gelfand-Naimark Theorem, the spectrum is canonically homeomorphic to $X$. Furthermore, the order generated by $C'$ is precisely the original order $\le$ of the compact ordered space $X$ (by complete order regularity of the latter). Thus, the resulting structure is identical to the starting structure: $(X, \mathscr{T}, \le)$.
\end{itemize}
\end{enumerate}
Since the mappings are mutually inverse (up to canonical isomorphism), the correspondence is established.
\end{proof}

\begin{proposition} \label{morph}
Let $(X,\mathscr{T}_X,\le_X)$ and $(Y,\mathscr{T}_Y,\le_Y)$ be compact ordered spa\-ces and let $\pi\colon C(Y,\mathbb{C})\to C(X,\mathbb{C})$ be a unital *-homomorphism. Then $\pi(g)=g\circ \varphi$ for a unique continuous map $\varphi\colon X\to Y$, and $\pi\big(C^+(Y,\le_Y,\mathbb{R})\big)\subset C^+(X,\le_X,\mathbb{R})$ if and only if $\varphi$ is isotone. Moreover, if $(B,C)$ and $(B',C')$ are pairs as in Theorem \ref{cpp} and $\pi\colon B\to B'$ is a unital *-homomorphism such that $\pi(C)\subset C'$, then the map $x'\mapsto x'\circ \pi$ from the character space of $B'$ to that of $B$ is continuous and isotone for the orders of Theorem \ref{cpp}.
\end{proposition}

\begin{proof}
Recall that the characters of $C(Y,\mathbb{C})$ for $Y$ compact Hausdorff are exactly the point evaluations $\delta_y\colon g\mapsto g(y)$, $y\in Y$.
The first statement is thus Gelfand duality \cite{murphy90}: $\varphi(x)$ is the point of $Y$ corresponding to the character $g\mapsto \pi(g)(x)$.
As for uniqueness, suppose $\varphi,\varphi'\colon X\to Y$ both satisfy $\pi(g)=g\circ\varphi=g\circ\varphi'$ for all $g$. Then $g(\varphi(x))=g(\varphi'(x))$ for every $g\in C(Y,\mathbb{C})$. A compact Hausdorff space is completely regular (Urysohn), so continuous functions separate its points. Hence $\varphi(x)=\varphi'(x)$ for every $x$. For continuity of $\varphi$, for each $g$ the function $g\circ\varphi=\pi(g)$ is continuous. The topology of a compact Hausdorff (hence Tychonoff) $Y$ is the weakest one making all $g\in C(Y,\mathbb{C})$ continuous, so $\varphi$ is continuous.

If $\varphi$ is isotone and $g\in C^+(Y,\le_Y,\mathbb{R})$, then $g\circ\varphi$ is non-negative, continuous and isotone. Conversely, suppose $\pi(C^+(Y,\le_Y,\mathbb{R}))\subset C^+(X,\le_X,\mathbb{R})$ and let $x\le_X x'$. Then $g(\varphi(x))\le g(\varphi(x'))$ for every $g\in C^+(Y,\le_Y,\mathbb{R})$, and since $\le_Y=\le_{C^+(Y,\le_Y,\mathbb{R})}$ (see the Introduction) we get $\varphi(x)\le_Y\varphi(x')$.

For the last statement, $x'\circ\pi$ is a unital multiplicative linear functional, hence a character of $B$, and the map $x'\mapsto x'\circ\pi$ is weak*-continuous because $x'\mapsto x'(\pi(f))$ is continuous for every $f\in B$. If $x'_1\le x'_2$ then, for every $f\in C$, $\pi(f)\in C'$ and so $x'_1(\pi(f))\le x'_2(\pi(f))$, that is, $x'_1\circ\pi\le x'_2\circ\pi$.
\end{proof}

Together with Theorem \ref{equivalence_theorem}, this shows that the correspondence is a dual equivalence between the category of compact ordered spaces, with continuous isotone maps, and the category of pairs $(B,C)$ as in Theorem \ref{cpp}, with unital *-homomorphisms $\pi$ satisfying $\pi(C)\subset C'$.

\section{\texorpdfstring{Bernstein-type}{Bernstein-type} realization of the rational approximation} \label{secber}
In this section we provide an explicit rational approximation of non-negative continuous isotone functions  over $[0,b]\subset \mathbb{R}$  whose existence was established in Theorem \ref{cptr}.

For $h:[0,1]\to\mathbb{R}$ we denote by
\[
(B_nh)(t)=\sum_{k=0}^n h\Big(\frac kn\Big)\binom nk t^k(1-t)^{n-k},\qquad t\in[0,1],
\]
the $n$-th Bernstein polynomial of $h$, and for a function $g$ defined on an interval $I$ we denote by
\[
\omega_g(\delta)=\sup\{|g(s)-g(t)|: s,t\in I,\ |s-t|\le\delta\}
\]
its modulus of continuity. We shall use two classical properties of Bernstein polynomials, whose short proofs we include for completeness \cite{lorentz86}.

\begin{lemma} \label{bernlem}
Let $h\colon[0,1]\to\mathbb{R}$ be continuous and let $n\ge 1$.
\begin{itemize}
\item[(a)] If $h$ is non-decreasing, then $B_nh$ is non-decreasing on $[0,1]$.
\item[(b)] $\sup_{t\in[0,1]}|(B_nh)(t)-h(t)|\le\frac54\,\omega_h\big(n^{-1/2}\big)$.
\end{itemize}
\end{lemma}

\begin{proof}
(a) Differentiating term by term and using
$k\binom nk=n\binom{n-1}{k-1}$, $(n-k)\binom nk=n\binom{n-1}{k}$, one obtains
\[
(B_nh)'(t)=n\sum_{k=0}^{n-1}\Big[h\Big(\frac{k+1}n\Big)-h\Big(\frac kn\Big)\Big]
\binom{n-1}k t^k(1-t)^{n-1-k},
\]
and every bracket is non-negative when $h$ is non-decreasing.

(b) Write $w_k(t)=\binom nk t^k(1-t)^{n-k}$ and recall the identities
\[
\sum_{k=0}^n w_k(t)=1,\qquad
\sum_{k=0}^n \frac kn\,w_k(t)=t,\qquad
\sum_{k=0}^n\Big(\frac kn-t\Big)^2w_k(t)=\frac{t(1-t)}n .
\]
Let $\delta>0$. For $\lambda\ge0$, denoting by $\lceil\lambda\rceil$ the {\em ceiling} of $\lambda$, that is, the smallest integer $\ge\lambda$ (so that $\lambda\le\lceil\lambda\rceil<\lambda+1$), we have
$\omega_h(\lambda\delta)\le\omega_h(\lceil\lambda\rceil\delta)
\le\lceil\lambda\rceil\,\omega_h(\delta)\le(1+\lambda)\,\omega_h(\delta)$.
Hence, for $s,t\in[0,1]$: if $|s-t|\le\delta$ then
$|h(s)-h(t)|\le\omega_h(\delta)$, while if $|s-t|>\delta$ then
$|h(s)-h(t)|\le(1+|s-t|/\delta)\,\omega_h(\delta)
\le(1+(s-t)^2/\delta^2)\,\omega_h(\delta)$. In all cases
\[
|h(s)-h(t)|\le\Big(1+\frac{(s-t)^2}{\delta^2}\Big)\omega_h(\delta),
\]
and therefore
\begin{align*}
|(B_nh)(t)-h(t)|&\le\sum_{k=0}^n\Big|h\Big(\frac kn\Big)-h(t)\Big|w_k(t)
\le\omega_h(\delta)\Big(1+\frac{t(1-t)}{n\delta^2}\Big)\\
&
\le\omega_h(\delta)\Big(1+\frac1{4n\delta^2}\Big).
\end{align*}
The choice $\delta=n^{-1/2}$ gives the claim.
\end{proof}

\begin{theorem} \label{thmbern}
Let $b>0$ and let $f\colon[0,b]\to\mathbb{R}$ be a non-negative, non-decreasing,
continuous function. For each integer $n\ge 1$ and $k=0,1,\dots,n$ set
\[
c_{n,k}=
\begin{cases}
\dfrac{bk}{n-k}, & 2k\le n,\\[2mm]
b, & 2k\ge n,
\end{cases}
\]
(the two expressions agree for $2k=n$), so that $c_{n,k}\in[0,b]$, and define
the rational function
\[
R_n(f;x)=\sum_{k=0}^n f(c_{n,k})\binom nk\frac{b^{\,n-k}x^k}{(b+x)^n},
\qquad x\in[0,\infty).
\]
Then:
\begin{enumerate}
\item $R_n(f;\cdot)$ is the ratio of two polynomials with non-negative
coefficients. If $f\not\equiv0$, numerator and denominator both have degree $n$.
\item $R_n(f;\cdot)$ is non-decreasing on $[0,\infty)$, in particular on $[0,b]$.
\item For every $n\ge1$,
\[
\sup_{x\in[0,b]}|R_n(f;x)-f(x)|\le\frac54\,\omega_f\Big(\frac{4b}{\sqrt n}\Big).
\]
In particular $R_n(f;\cdot)\to f$ uniformly on $[0,b]$ as $n\to\infty$.
\end{enumerate}
\end{theorem}

\begin{proof}
{\em Reformulation.} Let $p\colon[0,\infty)\to[0,1)$, $p(x)=x/(b+x)$, and let
$\psi\colon[0,1]\to[0,b]$ be given by
\[
\psi(t)=
\begin{cases}
\dfrac{bt}{1-t}, & t\in[0,\tfrac12],\\[2mm]
b, & t\in[\tfrac12,1].
\end{cases}
\]
Both maps are continuous and non-decreasing, $p$ being strictly increasing.
For $2k\le n$ we have $k/n\le 1/2$ and $\psi(k/n)=b(k/n)/(1-k/n)=bk/(n-k)$,
while for $2k\ge n$ we have $\psi(k/n)=b$; thus $\psi(k/n)=c_{n,k}$ for all $k$.
The function $h:=f\circ\psi\colon[0,1]\to\mathbb{R}$ is continuous, non-negative and
non-decreasing. Since $1-p(x)=b/(b+x)$,
\[
\binom nk p(x)^k\big(1-p(x)\big)^{n-k}=\binom nk\frac{b^{\,n-k}x^k}{(b+x)^n},
\]
and therefore
\begin{equation}\label{eq:BnRn}
R_n(f;x)=(B_nh)\big(p(x)\big),\qquad x\in[0,\infty).
\end{equation}
Moreover, $p([0,b])=[0,\tfrac12]$ and for $x\in[0,b]$
\[
\psi\big(p(x)\big)=\frac{b\,\frac{x}{b+x}}{\frac{b}{b+x}}=x,
\qquad\text{hence}\qquad
h\big(p(x)\big)=f(x),\quad x\in[0,b].
\]

{\em Part 1.} We have $R_n(f;x)=P_n(x)/(b+x)^n$ with
\[
P_n(x)=\sum_{k=0}^n f(c_{n,k})\binom nk b^{\,n-k}x^k.
\]
Since $f\ge0$, $b>0$ and the binomial coefficients are positive, all
coefficients of $P_n$ are non-negative, while
$(b+x)^n=\sum_{j=0}^n\binom nj b^{\,n-j}x^j$ has positive coefficients. The
coefficient of $x^n$ in $P_n$ is $f(c_{n,n})=f(b)$. Since $f$ is non-negative and non-decreasing, $0\le f\le f(b)$ on $[0,b]$, so $f(b)=0$ if and only if $f\equiv0$. Hence if $f\not\equiv0$ then $f(b)>0$ and
$\deg P_n=n=\deg (b+x)^n$, while if $f\equiv0$ then $P_n=0$ and $R_n(f;\cdot)=0$.

{\em Part 2.} By Lemma \ref{bernlem}(a), $B_nh$ is non-decreasing on $[0,1]$. As $p$ is increasing,
(\ref{eq:BnRn}) shows that $R_n(f;\cdot)=(B_nh)\circ p$ is non-decreasing on
$[0,\infty)$.

{\em Part 3.} The function $\psi$ is Lipschitz on $[0,1]$ with constant $4b$. Indeed, on $[0,\frac12]$ we have $\psi'(t)=b/(1-t)^2\le 4b$, on $[\frac12,1]$ the
function $\psi$ is constant, and since $\psi$ is continuous, for $0\le s<t\le1$ we
may split $[s,t]$ at $\frac12$ if necessary and apply the mean value theorem on
each piece. Since $\psi$ takes values in $[0,b]$, we get
$\omega_h(\delta)\le\omega_f(4b\delta)$ for every $\delta\ge 0$. Now let $x\in[0,b]$. By (\ref{eq:BnRn}), the identity $h(p(x))=f(x)$ and
Lemma \ref{bernlem}(b),
\[
|R_n(f;x)-f(x)|=\big|(B_nh)(p(x))-h(p(x))\big|
\le\frac54\,\omega_h\Big(\frac1{\sqrt n}\Big)
\le\frac54\,\omega_f\Big(\frac{4b}{\sqrt n}\Big).
\]
Since $f$ is uniformly continuous on the compact interval $[0,b]$,
$\omega_f(\delta)\to0$ as $\delta\to0$, which gives uniform convergence.
\end{proof}

\begin{remark}
Since $B_nh\ge\min h=f(0)$, we have $R_n(f;x)\ge f(0)$ for all $x\ge0$; thus
$R_n(f;\cdot)$ is positive whenever $f(0)>0$. If $f(0)=0$ one may use
$R_n(f;\cdot)+\frac1n=\big(nP_n(x)+(b+x)^n\big)/\big(n(b+x)^n\big)$, which is
positive, isotone, has non-negative coefficients, numerator and denominator
of degree $n$, and still converges uniformly to $f$ on $[0,b]$. In this way one obtains elements of the family $S$ of Theorem \ref{cptr}.
\end{remark}

\begin{remark}
The construction extends to any compact interval $[a,b]$, $a<b$, keeping the non-negativity of the coefficients. Let $f\colon[a,b]\to\mathbb{R}$ be non-negative, non-decreasing and continuous. If $a> 0$, extend $f$ to $[0,b]$ by setting $f(x)=f(a)$ for $x\in[0,a]$, and apply Theorem \ref{thmbern} on $[0,b]$. If $a<0$, apply Theorem \ref{thmbern} to $\tilde f(y)=f(y+a)$, $y\in[0,b-a]$, and consider $x\mapsto R_n(\tilde f;x-a)=R_n(\tilde f;x+|a|)$: its numerator and denominator are polynomials in $x+|a|$ with non-negative coefficients, hence polynomials in $x$ with non-negative coefficients; it is non-decreasing and converges uniformly to $f$ on $[a,b]$.
\end{remark}
%

\section{Conclusions}
We have established a general approximation theorem for isotone functions on a compact preordered space (Thm.\ \ref{besm2}). The theorem employs a family of functions $S$, where the core requirement is its invariance under a single transformation $\phi$ with specific fixed-point properties. This minimalist invariance condition provides considerable flexibility, allowing $\phi$ to be chosen as an algebraic expression.

As a concrete application, we obtained the uniform approximation on compact intervals of non-negative continuous isotone functions by isotone rational functions with non-negative coefficients (Thm.\  \ref{cptr}), complete with an explicit realization of Bernstein type (Sec.\ \ref{secber}).

Ultimately, this framework offers a new means to characterize function spaces that can be interpreted as families of continuous isotone functions over a compact set $X$ endowed with a closed order (Sec.\ \ref{secalg}).

Our primary motivation comes from spacetime geometry as we intend to apply this algebraic characterization to recover the causal structure of spacetime. In fact, we showed in previous work   that every stably causal spacetime (a time-oriented Lorentzian manifold) is quasi-uniformizable and can thus be densely embedded into a compact ordered space via Nachbin’s compactification \cite[Theorem 4.15]{minguzzi12d} \cite{minguzzi25}. This connection opens a potential pathway toward unifying General Relativity with Quantum Mechanics, a direction of investigation we plan to explore in future work.

{\small
\section*{\small Declaration of generative AI and AI-assisted technologies in the manu\-script preparation process}

In the preparation of this work, the author used DeepSeek-V3.2 and Google Gemini 3 Pro as aids in developing the results of Sec.\ \ref{secber} and, more generally, for checking the English throughout the text. The author has reviewed and edited the content as needed after using these tools and takes full responsibility for the published article's content.
}

\bigskip
\noindent{\small\textit{Author's mailing address:} Dipartimento di Matematica, Universit\`a degli Studi di Pisa, Via B. Pontecorvo 5, I-56127 Pisa, Italy.}

\end{document}